\documentclass{article}

\usepackage{epsfig}
\usepackage{latexsym,amsmath,amsfonts,amscd}
\usepackage{amssymb,multirow}
\usepackage{graphics}
\usepackage{epsfig}
\usepackage{subfigure}
\usepackage{verbatim}
\usepackage{epstopdf}
\usepackage{color,leftidx}
\usepackage{amsthm}
\usepackage{mathrsfs}

\usepackage{indentfirst}
\usepackage{textcomp}

\usepackage{graphicx}
\usepackage{booktabs, longtable}
\usepackage{titletoc}
\usepackage{graphics}
\usepackage{tabularx}
\usepackage[subfigure,caption2]{ccaption}
\usepackage{enumerate}
\usepackage{cite}
\usepackage{stmaryrd}  

\newcommand{\llkh}{ \{\!\!\{ }
\newcommand{\rrkh}{ \}\!\!\} }
\newcommand{\sanshu}{ |\!|\!| }

\newtheorem{theorem}{Theorem}[section]
\newtheorem{lemma}{Lemma}[section]
\newtheorem{example}{Example}[section]

 \newtheorem{assumption}{Assumption}[section]
  
\numberwithin{equation}{section}

\usepackage{fancyhdr}

\def\p{\partial}

\usepackage{authblk}

\begin{document}

\title{A New $L2-1_{\sigma}$-Interior Penalty Method for Variable-Order Time-Fractional Subdiffusion Interface Problem with Curved Interface}

\date{\today}

\author[1]{Hongying Huang}

\author[2]{Chanchan Hao}
\author[3]{Changmu Yu}
\author[1,*]{Huili Zhang}

\affil[1]{School of Arts and Sciences, Guangzhou Maritime University, Guangzhou, 510725, Guangdong, China. email: huanghy@lsec.cc.ac.cn}%

\affil[2]{School of Information Engineering, Zhejiang Ocean University, Zhoushan, 316000, Zhejiang, China. email: 15349124106@163.com} %

\affil[3]{Fuzhou  Vocational  Technical  College, Fuzhou, 344000, Jiangxi, China. email: 632581374@qq.com} 

\affil[*]{Corresponding author: zhang.huili0203@163.com}

\maketitle

\begin{abstract}
This paper treats variable-order time-fractional subdiffusion with discontinuous coefficients across a curved interface using $L2\!-\!1_\sigma$ time stepping on graded meshes and a symmetric interior penalty FEM on body-fitted meshes. Stability and optimal a priori error estimates in a discrete-in-time $L^2$ norm are established, yielding second-order temporal accuracy. While analysis typically assumes $\alpha_n$ at $t_{n-\sigma_n}$ lies in the range of $\alpha(t)$ on $[t_{n-1},t_n]$ and $\alpha_n\le \alpha(t_{n-\alpha_n/2})$, experiments indicate the second inequality can be relaxed or omitted, enabling straightforward selection of $\alpha_n$ from many admissible values without solving a nonlinear equation. Numerical results verify temporal rates $\min\{2,r\delta\}$, spatial order $\min\{s,k+1\}$, and robustness to superconvergent points and interface geometry.
\end{abstract}

%

{\bf Keywords:}{ $L2\!-\!1_\sigma$  formula, interior penalty method, interface problem, variable-order Caputo derivative, subdiffusion.}

\section{Introduction}\label{sec_1}

Suppose that $\Omega \subset \mathbb{R}^2$ is an open, bounded, polygonal domain. The interface $\Gamma$ is a closed curve that divides $\Omega$ into two non-overlapping subdomains: $\Omega^-$ and $\Omega^+$. Here, $\Omega^-$ denotes the interior subdomain enclosed by $\Gamma$, while $\Omega^+$ represents the exterior subdomain lying outside $\Gamma$. Thus, the closure of $\Omega$ satisfies $\overline{\Omega} = \overline{\Omega^+} \cup \overline{\Omega^-}$ (see Figure \ref{domain}). 
 Consider the following variable-exponent subdiffusion problem with discontinuous diffusion coefficients
\begin{eqnarray}\label{eq:original_eq} 
_{0}^{C}D_{t}^{\alpha(t)}u-\nabla\cdot\left(\beta (\mathbf{x})\nabla u\right)=f(\mathbf{x},t),& (\mathbf{x},t)\in\Omega ^+\cup \Omega ^- \times(0,T], \\
u(\mathbf{x},t)=g(\mathbf{x},t), & (\mathbf{x},t)\in\p\Omega\times(0,T), \label{eq:original_eq1} \\
\lbrack  u \rbrack=\phi, \lbrack  \beta \nabla u\cdot\mathbf{n} \rbrack=\psi, &  (\mathbf{x},t)\in\Gamma\times(0,T), \label{eq:original_eq2} \\
u(\mathbf{x},0)=u_0(\mathbf{x}), &\mathbf{x}\in\Omega, \label{eq:original_eq3} 
\end{eqnarray}
 where $\lbrack u\rbrack=u|_{\Omega^-}-u|_{\Omega^+}, \lbrack \alpha u_{\mathbf{n}}\rbrack=\alpha \nabla u|_{\Omega^-}\cdot\mathbf{n^-}+\alpha \nabla u|_{\Omega^+}\cdot\mathbf{n^+}$ with $\mathbf{n^-}$ being the unit outward normal vector on $\Gamma$ pointing from $\Omega^-$ to $\Omega^+$ and set $\mathbf{n}^+=-\mathbf{n}^-$.  Diffusion coefficient $\beta(\mathbf{x})$ is defined as,
 \begin{equation*} \beta(\mathbf{x})=\left\{
\begin{array}
{ll}\beta^-(\mathbf{x}), & \mathbf{x}\in\Omega^-, \\
\beta^+(\mathbf{x}), & \mathbf{x}\in\Omega^+.
\end{array}\right.\end{equation*} 

The Caputo differential operator $ _0^C D_t^{\alpha(t)}$ is defined by
\begin{equation*} _{0}^{C}D_t^{\alpha(t)}g(t)=\frac{1}{\Gamma(1-\alpha(t))}\int_0^t\frac{g^{\prime}(\xi)}{(t-\xi)^{\alpha(t)}}\mathrm{d}\xi,\end{equation*}
where $\alpha(t)\in [0,1)$ is the order of Caputo fractional differential operator.

\begin{figure}[hb]
   \centering
       \includegraphics[width=0.35\paperwidth]{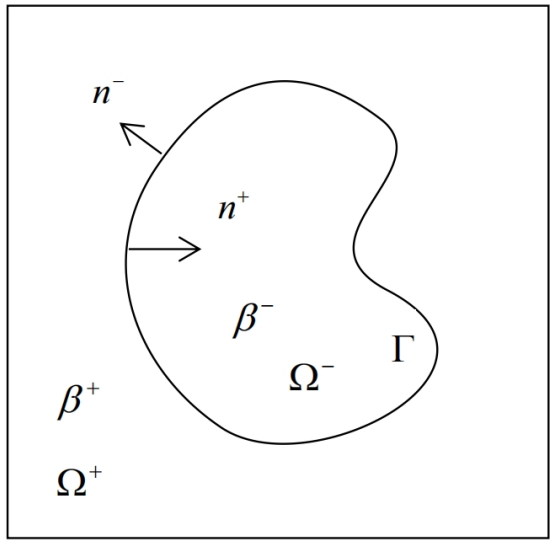}
       \caption{Sample domain $\Omega = \Omega^+\cup\Omega^-\cup\Gamma$. }
       \label{domain}
\end{figure}


In recent decades, time-fractional differential equations have been widely used to model complex phenomena in fields such as viscoelasticity, signal processing, and noise reduction \cite{podlubny1998fractional,metzler2000random,metzler2000subdiffusive}. However, increasing evidence suggests that variable-order fractional models are more suitable for systems with evolving memory and heterogeneous structures \cite{podlubny1998fractional,wang2019wellposedness,sun2019review}.

Many real-world transport processes exhibit nonlocality, memory dependence, and spatial-temporal heterogeneity that cannot be adequately described by constant-order models. In heterogeneous porous media, for example, nonuniform pore distributions lead to region-dependent diffusion behavior, naturally requiring space- or time-varying fractional orders. Similar interface-driven heterogeneity appears in anomalous diffusion in biological tissues, heat transfer in composites, groundwater transport, and lithium-ion diffusion in batteries, where parameters such as diffusion coefficients and fractional orders may change abruptly across interfaces. These features motivate the formulation of variable-order time-fractional interface models for accurately describing coupled transport across heterogeneous media.

Umarov and Steinberg proved the existence and uniqueness of solutions to variable-order time-fractional differential equations under the assumption that the fractional order is piecewise constant in time \cite{umarov2009variable}. Without this assumption, Wang et al. established the well-posedness of a variable-order linear time-fractional mobile/immobile transport equation, and showed that the solution is fully regular when the variable order has an integer limit, but may exhibit singular behavior when the initial-time order is non-integer \cite{wang2019wellposedness}. They subsequently extended these results to nonlinear variable-order time-fractional differential equations \cite{wang2019analysis}. More recently, Zheng \cite{zheng2025two} developed a convolution-based framework to study well-posedness, regularity, inverse problems, and numerical approximations for variable-order subdiffusion equations.

The numerical treatment of variable-order time-fractional differential equations is significantly more involved. Wang and Zheng \cite{wang2019analysis} proposed a graded-mesh finite difference method for a nonlinear variable-order time-fractional equation without spatial variables, recovering the optimal first-order convergence rate $O(\tau)$. Zheng and Wang \cite{zheng2021optimal} combined the $L1$ formula for the variable-order Caputo derivative with finite element spatial discretization, and obtained optimal first-order temporal accuracy. Ma et al. \cite{ma2023stabilizer} developed an $L1$-based fully discrete, stabilizer-free weak Galerkin finite element method for an initial-boundary value problem of variable-order Caputo time-fractional diffusion, also achieving first-order accuracy. Du et al. \cite{du2020temporal} employed the $L2\!-\!1_\sigma$  formula for the variable-order Caputo derivative, deriving a second-order temporal difference scheme together with a fourth-order spatial finite difference method for multidimensional variable-order subdiffusion equations. However, at each time step, a nonlinear equation must be solved by Newton iteration to determine the parameter $\sigma_n$ in the $L2\!-\!1_\sigma$  formula. Zhang et al. \cite{zhang2022fast} combined the $L2\!-\!1_\sigma$  formula \cite{du2020temporal} with exponential-sum approximation \cite{zhang2022exponential} to construct a fast second-order approximation for the variable-order Caputo derivative. Huang et al. \cite{huang2026determining} further proposed a second-order temporal accurate numerical scheme that couples $L2\!-\!1_\sigma$ temporal discretization with finite element spatial approximation. They also relaxed the criterion for selecting superconvergence points while preserving accuracy, thereby reducing the computational cost of determining these points. Additional computational methods for variable-order fractional nonlinear equations can be found in \cite{gu2023two,heydari2020cardinal,heydari2019computational}.

In contrast, relatively few results are available for numerical algorithms for time-fractional interface problems. In \cite{huang2020optimal}, a fully discrete local discontinuous Galerkin method was proposed for a time-fractional reaction-diffusion initial-boundary value problem with discontinuous diffusion coefficients, where the $L2\!-\!1_\sigma$  scheme on a graded temporal mesh was used for the time-fractional derivative. Chen et al. \cite{chen2022immersed} introduced an immersed finite element method for time-fractional diffusion equations with discontinuous coefficients, approximating the Caputo derivative by a nonuniform $L1$ scheme. However, both works focused on constant-order time-fractional interface models, and the interfaces involved are polygonal. Recently, Hao et al. \cite{hao2026interior} applied the $L1$ formula to discretize the variable-order fractional time derivative, and used a symmetric interior penalty method (IPM) on body-fitted meshes aligned with the curved interface $\Gamma$ for spatial discretization. When $\Gamma$ is curved, curved interface elements with one curved edge on $\Gamma$ are directly employed.  

To the best of our knowledge, numerical analysis for variable-order time-fractional interface problems remains relatively underdeveloped, especially for models with curved interfaces and discontinuous diffusion coefficients. Existing studies mainly focus on either variable-order problems without interfaces, or interface problems of constant order. In addition, the currently available variable-order interface schemes are mostly first-order accurate in time. 
Against this background, this paper investigates a variable-order time-fractional subdiffusion interface model with discontinuous coefficients across a curved interface. The interior penalty method is adopted for spatial discretization. In contrast to \cite{hao2026interior}, the $L2\!-\!1_\sigma$  scheme in \cite{huang2026determining} is employed for time discretization of the variable-order fractional derivative, yielding second-order temporal accuracy. Stability and optimal error estimates are rigorously established.

Throughout this paper, for any bounded domain $S\subset\Omega$ and any real number $s$, we denote by $H^s(S)$  the classical Sobolev space, equipped with norm $\|\cdot\|_{s,S}$ and seminorm $|\cdot|_{s,S}$. Let $S_1, S_2\subset\Omega$. We define
 \begin{equation*} H^s(S_1\cup S_2):=\{v\in L^2(\Omega): v|_{S_1}\in H^s(S_1), v|_{S_2}\in H^s({S_2})\}.\end{equation*} 
The  inner product and norm of $L^2(\Omega)$ are denoted by $(\cdot, \cdot)$ and $\|\cdot\|$, respectively.

Wang and Zheng \cite{wang2019wellposedness} established the well-posedness of a variable-order time-fractional mobile-immobile equation under the assumption that the diffusion coefficient $\beta(\mathbf{x})$ is continuous on $\Omega$. Their analysis also captures the initial singularity of the exact solution. However, to the best of our knowledge, there are still no well-posedness results for variable-order diffusion equations with interfaces, in particular for variable-order subdiffusion interface problems. Therefore, for the subsequent error analysis, we make the following assumption on the solution to the interface problem \eqref{eq:original_eq}-\eqref{eq:original_eq3}.

\begin{assumption}\label{assumption_a} Suppose that $\alpha\in C^1[0,T]$ and $0\leq \alpha(t)\leq\alpha^*<1$ on $[0,T]$ and that $\beta(\mathbf{x})\in C^1(\overline{\Omega^-}\cup\overline{\Omega^+}$ such that there exist two positive constants $\beta_l,\beta_u$ satisfying 
\begin{equation*} 0<\beta_l\leq\beta(\mathbf{x})\leq\beta_u,\quad\mathbf{x}\in\Omega.\end{equation*} 
The solution $u$ of the problem \eqref{eq:original_eq}-\eqref{eq:original_eq3}
 satisfies $u(\cdot,t)\in H^s(\Omega^-\cup\Omega^+), s>1$ for any $t\in [0,T]$ and
\begin{equation}\label{eq:assump-solution}
\|u(\cdot,t)\|+\|_0^CD_t^{\alpha(t)}u(\cdot,t)\|\leq Q_0,\quad \|\p^l_tu(\cdot,t)\|\leq Q(1+t^{\delta-l})\ \
 \text{for}\ \ l=0,1,2,3,
\end{equation}
where $\delta\in (0,1)$ is a positive constant. 
\end{assumption}

The rest of the article is organized as follows. In Section 2, we introduce the temporal discretization of the problem \eqref{eq:original_eq}, and the $L2\!-\!1_\sigma$  scheme is employed. In Section 3, we present the fully discrete scheme. The stability  and error estimate of the fully discrete scheme is analyzed in Section 4.  Finally, some numerical experiments are carried out to confirm the theoretical predictions established in this work.

\section{Temporal discretization}\label{sec3}

In this section, following the ideas of \cite{huang2026determining}, we employ the $L2\!-\!1_\sigma$  formula to approximate the variable-order fractional derivative. To handle the initial singularity, a temporal graded mesh is adopted. This leads to a semi-discrete scheme, and a truncation error estimate of order $3-\alpha_n^*$ is derived.

 We partition the interval $[0,T]$ into $N$  graded time steps and define $t_n=T(n/N)^r$, $0\leq n\leq N, r\geq 1$, with
\begin{equation}\label{eq:taun_rhon_def}
\tau_n:=t_n-t_{n-1},  \  \tau:=\max_{1\leq k\leq N}\tau_k=\tau_N,\   \rho_k:=\tau_k/\tau_{k+1},  \  \rho:=\max_{1\leq k\leq N-1}\rho_k.\end{equation}
  
 To define the evaluation point (i.e., the superconvergent point)  at each time step, we set
\begin{equation*}t_{n-\sigma_n}:=\sigma_n t_{n-1}+(1-\sigma_n)t_n, \end{equation*} 
and, following \cite{huang2026determining}, choose $\sigma_n=\alpha_n/2$, with $\alpha_n$ required to satisfy
\begin{equation}\label{sigma_cond} \alpha\!\left(t_{n-\frac{\alpha_n}{2}}\right)\ge \alpha_n,\qquad \alpha_n\in\left[\min_{t\in[t_{n-1},t_n]}\alpha(t),\ \max_{t\in[t_{n-1},t_n]}\alpha(t)\right]. \end{equation}

In previous $L2\!-\!1_{\sigma}$ scheme, $\alpha_n$ is typically computed at each time step by applying Newton's method to solve the nonlinear equation $\alpha(t_{n-\alpha_n/2})=\alpha_n$ (see \cite{du2020temporal}). By contrast, the strategy adopted here avoids such nonlinear solves and  allows greater flexibility in the choice of $\alpha_n$. For instance, $\alpha_n=\min_{t\in[t_{n-1},t_n]}\alpha(t)$ always satisfies \eqref{sigma_cond}. Hence, if $\alpha(t)$ is monotone (not necessarily linear), this value of $\alpha_n$ can be obtained directly without extra computation. Numerical experiments further suggest that the inequality $\alpha(t_{n-\alpha_n/2})\ge \alpha_n$ can be relaxed or even omitted. Consequently, for a general variable order $\alpha(t)$, $\alpha_n$ can be selected straightforwardly from a large family of admissible values.

Denote $\alpha_n^*:=\alpha(t_{n-\sigma_n})$. Evaluating \eqref{eq:original_eq} at $t=t_{n-\sigma_n}$,  we obtain
\begin{equation}\label{eq:tn_alphan00}
  ^C_0D_t^{\alpha_n^*}u(\mathbf{x},t_{n-\sigma_n})
-\nabla\cdot\left(\beta (\mathbf{x})\nabla u(\mathbf{x},t_{n-\sigma_n})\right)=f(\mathbf{x},t_{n-\sigma_n}).
\end{equation}

Following \cite{huang2026determining}, we apply the $L2\!-\!1_\sigma$ formula to discretize the Caputo time-fractional derivative, which yields
\begin{eqnarray}
(D_{\tau}^{\alpha_n^*}u)^{n-\sigma_n}-\nabla\cdot(\beta\nabla u(\mathbf{x},t_{n-\sigma_n}))=
f(\mathbf{x},t_{n-\sigma_n})+R^n,\quad 1\leq n\leq N,\label{eq:semi_time_discrete}
\end{eqnarray}
where 
\begin{equation*}R^n:=(D_{\tau}^{\alpha_n^*}u)^{n-\sigma_n}-^C_0D_t^{\alpha_n^*}u(\mathbf{x},t_{n-\sigma_n}),\end{equation*} 
 and
\begin{eqnarray}
(D_{\tau}^{\alpha_n^*}u)^{n-\sigma_n}
:=a_0^{(\alpha_n^*)}\triangledown_{\tau}u^n+\sum_{k=1}^{n-1}\left(a_{n-k}^{(\alpha_n^*)}\triangledown_{\tau}u^k
-b_{n-k}^{(\alpha_n^*)}\triangledown_{\tau}u^k+\rho_kb_{n-k}^{(\alpha_n^*)}\triangledown_{\tau}u^{k+1}\right)=\sum_{k=1}^{n} c_{n-k,n}^{(\alpha_n^*)}\triangledown_{\tau}u^k.\label{eq:Caputo_v_disc} 
\end{eqnarray}
Here, $\triangledown u_{\tau}^k:=u^k-u^{k-1}$, $u^k:=u(\mathbf{x},t_k)$, and
\begin{equation}
 a_{n-k}^{(\alpha_n^*)}:=\frac{1}{\tau_k}\int_{t_{k-1}}^{\min\{t_k, t_{n-\sigma_n}\}}\frac{\left(t_{n-\sigma_n}-\xi\right)^{-\alpha_n^*}}{\Gamma(1-\alpha_n^*)}\mathrm{d}\xi,\  0\leq k\leq n,\label{eq:ak_def}\end{equation}
\begin{equation}b_{n-k}^{(\alpha_n^*)}:=\frac{2}{\tau_k(\tau_k+\tau_{k+1})}\int_{t_{k-1}}^{t_k}\frac{(\xi-t_{k-1/2})(t_{n-\sigma_n}-\xi)^{-\alpha_n^*}}
{\Gamma(1-\alpha_n^*)}\mathrm{d}\xi,\  1\leq k\leq n-1, \label{eq:bk_def}
\end{equation}
and
\begin{equation}\label{eq:ckn_def}
{c}_{n-k,n}^{(\alpha_n^*)}:=\left\{\begin{array}{ll}
a_0^{(\alpha_n^*)},\quad & \text{for}\ \  k=n=1,\\
a_0^{(\alpha_n^*)}+\rho_{n-1}b_1^{(\alpha_n^*)}, & \text{for}\ \ k=n\geq 2,\\
a_{n-k}^{(\alpha_n^*)}+\rho_{k-1}b_{n-k+1}^{(\alpha_n^*)}-b_{n-k}^{(\alpha_n^*)}, &   \text{for}\ \ 2\leq k\leq n-1,\ \ 2\leq n,\\
a_{n-1}^{(\alpha_n^*)}-b_{n-1}^{(\alpha_n^*)},  &  \text{for}\ \  k=1,\ \  2\leq n.\end{array}\right.
\end{equation}

Applying the notation of \eqref{eq:ckn_def}, we introduce  the  complementary discrete kernels $\mathbb{P}_j^{(n)}$ by
\begin{equation*}\label{eq:P_def}
\sum_{k=m}^n\mathbb{P}_{n-k}^{(n)}c_{k-m,k}^{(\alpha_k^*)}=1,\qquad 1\leq m\leq n\leq N.
\end{equation*}
Accordingly, these kernels can be constructed recursively as \cite{huang2023alpha,liao2019discrete}:
\begin{equation*}\mathbb{P}_0^{(n)}=\frac{1}{c_{0,n}^{(\alpha_n^*)}},\ \mathbb{P}_{j}^{(n)}=\frac{1}{c_{0,m}^{(\alpha_n^*)}}\sum_{k=0}^{j-1}\mathbb{P}_k^{(n)}
\left(c_{j-k-1,n-k}^{(\alpha_{n-k}^*)}-c_{j-k,n-k}^{(\alpha_{n-k}^*)}\right), \end{equation*} 
for $1\leq j\leq n-1$ and $\ 1\leq n\leq N$.

We now present several auxiliary results used in the subsequent analysis; for detailed proofs, we refer the reader to \cite{huang2026determining}. The following lemmas provide  local truncation and  consistency error estimates for the $L2\!-\!1_\sigma$  approximation. They also include the key positivity estimate, a discrete Gronwall inequality, and kernel-summation bounds required for the stability and convergence analysis.

\begin{lemma}\label{lem:v_format1}
Suppose that $\|\p_t^lv(\cdot,t)\|\leq Q(1+t^{\delta-l})$ with $\delta\in (0,1)$ and $l=0,1,2$. Let $\sigma_n\in [0,1/2]$ and $\alpha_n^*:=\alpha(t_{n-\sigma_n})$. Then
\begin{equation}\label{eq:localerror0}
\|v(\cdot,t_{n-\sigma_n})-v^{n-\sigma_n}\|\leq\left\{\begin{array}{ll}
C_1t_{n-\sigma_n}^{-\alpha_n^*}N^{-2},& r\geq {2}/{(\delta+\alpha_n^*)},\\
C_2t_{n-\sigma_n}^{-\alpha_n^*}N^{-r(\delta+\alpha_n^*)}\leq C_2t_{n-\sigma_n}^{-\alpha_n^*}N^{-r\delta},&  1\leq r<{2}/{(\delta+\alpha_n^*)},
\end{array}\right.
\end{equation}
where $v^{n-\sigma_n}:=\sigma_nv^{n-1}+(1-\sigma_n)v^n$.
\end{lemma}

 The next lemma gives the consistency error of the Caputo derivative discretization.

\begin{lemma}\label{lem:CaputoEst1}
Suppose that $\|\p_t^lv(\cdot,t)\|\leq Q(1+t^{\delta-l})$ with $\delta\in (0,1)$ and $l=0,1,2,3$. Let $\sigma_n:=\alpha_n/2$ and $\alpha_n^*:=\alpha(t_{n-\sigma_n})$, where $\alpha_n\in [\min_{t\in[t_{n-1},t_n]}\alpha(t),$ $\max_{t\in[t_{n-1},t_n]}\alpha(t)]$.
  Then
\begin{equation}\label{eq:CaputoEst1}
\left\|^C_0D_t^{\alpha_n^*}v(\cdot,t_{n-\sigma_n})-(D_{\tau}^{\alpha_n^*}v)^{n-\sigma_n}\right\|\leq C_3t_{n-\sigma_n}^{-\alpha_n^*}
N^{-\min\{3-\alpha^*, r\delta\}},\quad 1\leq n\leq N.
\end{equation}
\end{lemma}
 
For the stability argument, we also need the following coercivity-type inequality for the discrete fractional operator.

\begin{lemma}\label{thm:dtaualpha}
Let the parameter $\sigma_n=\alpha_n/2, \alpha_n\in [0,1)$, and $\alpha_n^*=\alpha(t_{n-\sigma_n})$ such that $\alpha_n^*\geq\alpha_n$.  Let the sequence $\{v^k\}_{k=0}^N\subset L^2(\Omega)$, and define $v^{n-\sigma_n}:=\sigma_nv^{n-1}+(1-\sigma_n)v^n$.   Then  
\begin{equation}\label{eq:ckn-v}
\left((D_{\tau}^{\alpha_n^*}v)^{n-\sigma_n},\, v^{n-\sigma_n}\right)\geq \frac{1}{2}\sum_{k=1}^{n}c_{n-k,n}^{(\alpha_n^*)}\left(\|v^k\|^2-\|v^{k-1}\|^2\right),\quad 2\leq n\leq N.\end{equation}
\end{lemma}

 Combining the above inequality with a suitable discrete Gronwall argument yields the following estimate.

\begin{lemma}\label{lem:discrete_gronwall}Suppose that the nonnegative sequences $\{\xi^k\}_{k=1}^n$ and $\{\eta^k\}_{k=1}^n$ are bounded, and the grid function $\{\zeta^n\}_{n=0}^N$ satisfies
\begin{equation*}(D_{\tau}^{\alpha_n^*}\zeta^2)^{n-\sigma_n}:=\sum_{k=1}^{n} c_{n-k,n}^{(\alpha_n^*)}\triangledown_{\tau}(\zeta^k)^2\leq \xi^n\zeta^{n-\sigma_n}+(\eta^n)^2, \ \text{for}\ \ n\geq 1.\end{equation*} 
Then 
\begin{equation*}\zeta^n\leq \zeta^0+\max_{1\leq k\leq n}\sum_{j=1}^k\mathbb{P}_{k-j}^{(k)}(\xi^j+\eta^j)+\max_{1\leq j\leq n}\eta^j,\ \ \text{for}\ \ 1\leq n.\end{equation*} 
\end{lemma}

To bound the convolution terms involving $\mathbb{P}_{n-j}^{(n)}$, we further use the following  kernel estimate.
%

\begin{lemma}\label{lem:Pjn_t0}
Setting $l_N=1/\ln N$, one has
\begin{equation}\label{eq:Pjn_tj}
\sum_{j=1}^n\mathbb{P}_{n-j}^{(n)}t_j^{-\alpha_j^*}\leq \frac{(1+2^r)\exp(r)
\max_{1\leq j\leq n}\Gamma(1+l_N-\alpha_j^*)}{\Gamma(1+l_N)},\ \ \text{for}\ \ 1\leq n,
\end{equation}
\begin{equation}\label{eq:Pjn_t0}
\sum_{j=1}^n\mathbb{P}_{n-j}^{(n)}\leq \frac{(1+2^r)\exp(r)t_n^{\alpha^*}
\max_{1\leq j\leq n}\Gamma(1+l_N-\alpha_j^*)}{\Gamma(1+l_N)},\ \ \text{for}\ \ 1\leq n.
\end{equation}
\end{lemma}

\section{fully discretization}\label{sec4}

In this section, we apply the interior penalty method to approximate the spatial derivative term in problem \eqref{eq:semi_time_discrete}, and thus derive a fully discrete scheme.

In what follows, we adopt the notation in  \cite{huang2020high}. Let $\mathcal{T}_h$ be a body-fitted, shape-regular triangulation satisfying $\overline{\Omega}=\overline{\Omega^-}\cup\overline{\Omega^+}= \cup_{{K\in \mathcal{T}_h}} K$. Each element $K\in\mathcal{T}_h$ lies entirely in either $\Omega^+$ or $\Omega^-$, and each triangle has at most two vertices on the interface $\Gamma$. For each $K\in\mathcal{T}_h$, let $h_K$ denote its diameter, and define $h:=\max_{K\in\mathcal{T}_h}h_K$.  The mesh is partitioned  into three categories:
\begin{eqnarray}\label{partition}&&
\mathcal{T}^+_h:=\{K\in\mathcal{T}_h: K\subset\Omega^+, \text{has at most one vertex on}\ \Gamma\},\nonumber\\
&&\mathcal{T}^-_h:=\{K\in\mathcal{T}_h: K\subset\Omega^-, \text{has at most one vertex on}\ \Gamma\},\\
&&\mathcal{T}^{\Gamma}_h:=\{K\in\mathcal{T}_h: K \text{\ has two vertices on}\ \Gamma\}.\nonumber
\end{eqnarray}

We classify mesh edges as follows. A boundary edge  is an edge in $\p K\cap\p \Omega$, where the element $K$ has two vertices on $\p\Omega$. An interface edge is an edge in $\p K\cap\Gamma$, where $K$ has two vertices on $\Gamma$. All remaining edges  are interior edges. Let $\mathcal{E}_h$ be the set of all edges of $\mathcal{T}_h$, and define $\mathcal{E}_h^I$, $\mathcal{E}_h^D$, and $\mathcal{E}_h^{\Gamma}$ as the sets of interior, boundary, and interface edges, respectively. Then $\mathcal{E}_h=\mathcal{E}_h^I\cup\mathcal{E}_h^D\cup\mathcal{E}_h^{\Gamma}$. For each edge $e$, let $h_e$ denote its length. We assume standard mesh regularity: there exists a positive constant $C_h$ such that, for all $K\in\mathcal{T}_h$, and all $e\in\mathcal{E}_h$, we have $h_K\leq C_h h_e$.

If $\Gamma$ is  curved and an element $K$ has two vertices on $\Gamma$, then $K$ is a curved triangle with one curved edge. In this work, such curved triangles are directly treated as interface elements. Following  \cite{cangiani2018adaptive,cangiani2021version,hao2026interior}, we next state several regularity assumptions needed for the stability and convergence analysis.

\begin{assumption}  \label{assumption_b}
    For all interface  elements $K\in \mathcal{T}_h^{\Gamma}$,  we assume:
    \begin{enumerate}[(a)]
\item \label{assump_a}\textbf{(Star-shapedness)} Each element $K$ with interface face
$e\subset\Gamma$ is star-shaped with respect to every vertex opposite $e$.
Moreover, $K$ is also star-shaped with respect to  the midpoints of all edges that share a common vertex with $e$ but are not edges of $e$ itself.

\item \label{assump_b}\textbf{(Shape regularity)}  
Uniformly over the mesh, for $\mathbf{m}(\mathbf{x})=\mathbf{x}-\mathbf{x}_0$ with $\mathbf{x}\in e$, where $\mathbf{x}_0$ is any vertex opposite $e\subset\Gamma$, and $\mathbf{n}(\mathbf{x})$ is the unit outward normal to $e$ at $\mathbf{x}$ pointing outward from $K$, we have
\begin{equation*}
\mathbf{m}(\mathbf{x})\cdot \mathbf{n}(\mathbf{x}) \ge C_r\,|\mathbf{m}(\mathbf{x})|,
\end{equation*}
and $|\mathbf{m}(\mathbf{x})| \sim h_K$ uniformly.

\item \label{assump_c}\textbf{(Piecewise $C^1$ smoothness)}  
The interface $\Gamma$ is the union of finitely many closed $C^1$ surfaces.

\item \label{assump_d}\textbf{(Interior-ball condition)}  
Each element $K$ is star-shaped with respect to a ball $B(\mathbf{x},\rho_K)\subset K$, where $\mathbf{x}\in K$, and $h_K\sim \rho_K$.
\end{enumerate}
Here, $a\sim b$ means that there exist positive constants $c_1,c_2$ such that
$c_1 a \le b \le c_2 a$.
\end{assumption}

We define the broken Sobolev space of order $s$ on the mesh $\mathcal{T}_h$ by
\begin{equation*}H^s(\mathcal {T}_h):=\{ v\in L^2(\Omega):\,v|_K \in H^s(K),\forall K\in \mathcal{T}_h\}.\end{equation*}
Let $P_k(E)$ denote the space of polynomials of degree at most $k$ on an element $E$. The piecewise polynomial space is defined as
\begin{equation*}\mathbb{V}_{h}^{k}:=\{v\in L^{2}(\Omega):v|_{K}\in P_{k}(K),\forall K\in\mathcal{T}_{h}\}.\end{equation*}

Let $K\in\mathcal{T}_h$ be an element with edge $e$, and suppose there exists $K'\in\mathcal{T}_h$ such that $e=\p K\cap\p K'$, i.e., $e$ is a common edge of $K$ and $K'$. Denote by $\mathbf{n}=(n_1,n_2)^T$  the unit normal vector on $\p K$ pointing from $K$ to $K'$. For any  $v\in H^1(K\cup K')$,  the average and the jump of $v$ across  $e$ are defined by
\begin{equation*}
		\llkh v \rrkh := \frac{1}{2}( v|_{K}+v|_{K'}),\quad
		\llbracket v\rrbracket := v|_{K'}-v|_K, \quad \text{on}\ e=\p K \cap \p K'.
	\end{equation*}
If $e=\p K\cap\p\Omega$ is a  boundary edge, we set
 \begin{equation*}\llkh v \rrkh:= v|_K,\quad \llbracket v\rrbracket := -v|_K.\end{equation*}
 
For $v\in \mathbb{V}_h^k$, we define the energy norm by
	\begin{equation*}
		\label{tripleNorm}
		\interleave v \interleave_h:=\left(\sum_{K \in \mathcal{T}_h} \int_K\beta \nabla v \cdot \nabla v \mathrm{d}\mathbf{x}+\sum_{e \in \mathcal{E}_h}\int_e \frac{\beta_0}{h_e} \llbracket v \rrbracket^2 \mathrm{d} s\right)^{1/2},
	\end{equation*}
where  $\beta_0$ is the penalty coefficient appearing in \eqref{eq:B_def}.

Let $v$ be an arbitrary smooth test function. Multiplying \eqref{eq:semi_time_discrete} by $v$ and applying integration by parts on each element $K$, we obtain
\begin{equation}\label{eq3}
\begin{aligned}
\int_{K}(D_{\tau}^{\alpha_n^*}u)^{n-\sigma_n}v\mathrm{d}\mathbf{x}+\int_{K}(\beta \nabla u^{n-\sigma_n})\cdot\nabla v\mathrm{d}\mathbf{x}-\int_{\p K} v (\beta\nabla u^{n-\sigma_n})\cdot \mathbf{n} \mathrm{d} s\\
=\int_{K}f(\mathbf{x},t_{n-\sigma_n})v\mathrm{d}\mathbf{x}+\int_{K}(E^n+R^n)v\mathrm{d}\mathbf{x},
\end{aligned}
\end{equation}
where \begin{equation*}E^n:=\nabla\cdot(\beta\nabla u(\mathbf{x},t_{n-\sigma_n}))- \nabla\cdot(\beta\nabla u^{n-\sigma_n}(\mathbf{x})).\end{equation*}

Summing \eqref{eq3} over all elements $K\in\mathcal{T}_h$ and adding the penalty terms, we arrive at the IPM weak formulation: find $u\in H^1(\Omega^+\cup\Omega^-)$  such that, for all $v\in H^1(\Omega^+\cup\Omega^-)$,
\begin{equation}\label{eq_whole_discrete0} 
((D_{\tau}^{\alpha_n^*}u)^{n-\sigma_n},v)+\mathbb{B}_h(u^{n-\sigma_n},v)=F^n(v)+(R^n+E^n,v),  
\end{equation}
where  
\begin{equation}\label{eq:B_def}
\begin{aligned}
\mathbb{B}_h(w,v)&:=\sum_{K\in\mathcal{T}_h}\int_K\beta\nabla w\cdot\nabla v \mathrm{d}\mathbf{x}+\sum_{e\in\mathcal{E}_h^I\cup\mathcal{E}_h^\Gamma}\int_e\left(\frac{\beta_0}{h_e} \llbracket w \rrbracket \llbracket v \rrbracket+\llkh \beta w_\mathbf{n}\rrkh \llbracket v \rrbracket+\llkh \beta v_\mathbf{n}\rrkh  \llbracket w \rrbracket\right)\mathrm{d} s\\
&+\sum_{e\in\mathcal{E}_h^D}\int_e\left(\frac{\beta_0}{h_e}wv-\beta v_\mathbf{n}w-\beta w_\mathbf{n}v\right)\mathrm{d} s,\\
\end{aligned}
\end{equation}
The right-hand side functional $F^n(v)$ is defined by
\begin{equation}
\begin{aligned}
F^n(v) & :=\sum_{e\in\mathcal{E}_h^\Gamma}\int_e\left(\frac{\beta_0}{h_e}\phi^{n-\sigma_n} \llbracket v \rrbracket+\llkh \beta v_{\mathbf{n}}\rrkh \phi^{n-\sigma_n}+\psi^{n-\sigma_n}\llkh v\rrkh \right)\mathrm{d} s \\
 & +\sum_{e\in\mathcal{E}_h^D}\int_e\left(\frac{\beta_0}{h_e}g^{n-\sigma_n}v-\beta v_\mathbf{n}g^{n-\sigma_n}\right)\mathrm{d} s+\int_\Omega f(\mathbf{x},t_{n-\sigma_n})v\mathrm{d}\mathbf{x}.
\end{aligned}
\end{equation}

By dropping the local truncation terms  $R^n$ and $E^n$ in \eqref{eq_whole_discrete0}, we obtain the full discrete form for \eqref{eq:original_eq} at $n$-th time step: find $u_h^n \in\mathbb{V}_{h}^{k}$ such that 
\begin{equation}\label{eq_whole_discrete}
\left\{\begin{array}{l}
((D_{\tau}^{\alpha_n^*}u_h)^{n-\sigma_n},v_h)+\mathbb{B}_h(u_h^{n-\sigma_n},v_h)=F^n(v_h),\quad \forall v_h\in\mathbb{V}_{h}^{k},\\
 \mathbb{B}_h(u_h^0,v_h)=\mathbb{B}_h(u_0,v_h),\quad \forall v_h\in \mathbb{V}_h^k.
\end{array}\right.
\end{equation}

Following \cite{hao2026interior}, under suitable geometric conditions on the interface $\Gamma$, the bilinear form $\mathbb{B}_h(\cdot,\cdot)$ is coercive and continuous, and $F^n(\cdot)$ is a bounded linear functional (see Lemma \ref{lem:bilinear_pro}). In addition, a discrete Poincar\'e inequality holds (see Lemma \ref{lem:poincare}).

 \begin{lemma}\label{lem:bilinear_pro}
Suppose that $\Gamma$ satisfies Assumption \ref{assumption_b}. Then the following properties hold for $\mathbb{B}_h(\cdot,\cdot)$ and $F^n(\cdot)$:
\begin{enumerate}[\bfseries (1)]
\item If the penalty parameter $\beta_0$ is sufficiently large, then there exists a positive constant $C_s$, independent of $v$, such that
\begin{equation}\label{eq:B_coercivity}
\mathbb{B}_h(v,v)\ge C_s \sanshu v \sanshu_h^2,\qquad \forall v\in\mathbb{V}_h^k.
\end{equation}
 
\item 
There exists a positive constant $C_b$, independent of $v$ and $w$, such that
\begin{equation}    
|\mathbb{B}_{h}(w,v)|\leq C_b\sanshu v\sanshu_h\sanshu w\sanshu_h,\quad \forall w, v\in\mathbb{V}_{h}^{k}. 
\end{equation}

\item Suppose that  $f(\cdot,t)\in L^2(\Omega)$ for all $t\in[0,T]$.
   Let \begin{equation}\label{eq:Gn}G^{n-\sigma_n}:= \left(\sum_{e\in\mathcal{E}_h^{\Gamma}}\int_e\frac{1}{h_e}{(\phi^{n-\sigma_n})^2}+(\psi^{n-\sigma_n})^2\mathrm{d} s+\sum_{e\in\mathcal{E}_h^D}\frac{1}{h_e}\int_e{(g^{n-\sigma_n})^2}\mathrm{d} s\right)^{1/2}.\end{equation}
    Then there exists a positive constant $C_4$  such that, for any $v\in\mathbb{V}_{h}^{k}$
\begin{equation}\label{eq:Fv}
|F^n(v)|\leq C_4G^{n-\sigma_n}\sanshu v\sanshu_h 
+\|f(t_{n-\sigma_n})\|\|v\|.\end{equation}
\end{enumerate}

\end{lemma}

 \begin{lemma}\label{lem:poincare}
Suppose that $\Gamma$ satisfies Assumption \ref{assumption_b}. For all $v\in H^s(\mathcal{T}_h)$, there exists a positive constant $C_5$, depending only on the geometry of $\Omega^+$, $\Omega^-$, and the mesh shape regularity, such that
\begin{equation}\label{eq:poincare}\|v\|\leq C_5\left(\sum_{E\in\mathcal{T}_h}\|\nabla v\|_{L^2(E)}^2+\sum_{e\in\mathcal{E}_h}\frac{1}{h_e}\|\llbracket v \rrbracket\|_{L^2(e)}^2\right)^{1/2}.\end{equation}

\end{lemma}

\section{Stability and error analysis}

In this section, we investigate the stability of the fully discrete scheme \eqref{eq_whole_discrete} and establish optimal error estimates for the numerical solution $u_h^n$.

For any $w\in H^s(\Omega^+\cup\Omega^-)$, $s>3/2$, define $\Pi_h w\in\mathbb{V}_h^k$ by
\begin{equation}\label{eq:elliptic_project}
\mathbb{B}_h(w-\Pi_h w,v)=0,\qquad \forall v\in\mathbb{V}_h^k.
\end{equation}
By the results in \cite{riviere2008discontinuous,huang2020high,bai2023coupling}, together with the continuity and coercivity of $\mathbb{B}_h(\cdot,\cdot)$, problem \eqref{eq:elliptic_project} admits a unique solution $\Pi_h w\in\mathbb{V}_h^k$, and the following estimate holds:
\begin{equation}\label{eq:projection_error_L2}
\|w-\Pi_h w\|
\le C_6 h^{\min\{k+1,s\}}
\|w\|_{H^s(\Omega^+\cup\Omega^-)}.
\end{equation}

\begin{theorem}\label{thm:stable}
Suppose that  $f(\cdot,t)\in L^2(\Omega)$ for all $t\in[0,T]$ and that $\phi, \psi$, and $g$ satisfy, for some positive constant $C_G$,
 \begin{equation*}G^{n-\sigma_n}\leq C_G,\qquad 1\le n\le N.\end{equation*} 
Let $\theta_n=\alpha_n/2$, $\alpha_n^*=\alpha(t_{n-\theta_n})$, and assume $0\le\alpha_n\le\alpha_n^*<1$ for $1\le n\le N$. Let $\{u_h^n\}_{n=0}^N$ be the solution of \eqref{eq_whole_discrete}. Then, for $1\le n\le N$,
\begin{equation}\label{eq:stable}
\|u_h^n\|
\le
\|u_h^0\|
+ C_7\max_{1\le j\le n}\|f(t_{j-\sigma_j})\|
+\left(C_7+\frac{C_4}{2\sqrt{C_s}}\right)\max_{1\le j\le n}G^{j-\sigma_j},
\end{equation}
where $G^{j-\sigma_j}$ is defined in \eqref{eq:Gn}, and
\begin{equation*}
C_7:=
2(1+2^r)\exp(r) t_n^{\alpha^*}\max\!\left\{1,\frac{C_4}{2\sqrt{C_s}}\right\}
\frac{\max_{1\le j\le n}\Gamma(1+1/\ln N-\alpha_j^*)}{\Gamma(1+1/\ln N)}.
\end{equation*}
\end{theorem}

\begin{proof}
Taking $v_h=u_h^{n-\sigma_n}$ in \eqref{eq_whole_discrete}, and using \eqref{eq:ckn-v}, \eqref{eq:B_coercivity}, and \eqref{eq:Fv}, we obtain
\begin{equation}\label{eq:stable1}
\frac{1}{2}\sum_{k=1}^n c_{n-k,n}^{(\alpha_n^*)}\big(\|u_h^k\|^2-\|u_h^{k-1}\|^2\big)
+ C_s\interleave u_h^{n-\sigma_n}\interleave_h^2
\le
C_4
G^{n-\sigma_n}\interleave u_h^{n-\sigma_n}\interleave_h
+\|f(t_{n-\sigma_n})\|\,\|u_h^{n-\sigma_n}\|.
\end{equation}
By Young's inequality,
\begin{equation*}
C_4 G^{n-\sigma_n}\interleave u_h^{n-\sigma_n}\interleave_h
\le
C_s\interleave u_h^{n-\sigma_n}\interleave_h^2
+\frac{C_4^2}{4C_s}\big(G^{n-\sigma_n}\big)^2.
\end{equation*}
Substituting this into \eqref{eq:stable1} yields
\begin{equation}\label{eq:stable2}
\frac{1}{2}\sum_{k=1}^n c_{n-k,n}^{(\alpha_n^*)}\big(\|u_h^k\|^2-\|u_h^{k-1}\|^2\big)
\le
\frac{C_4^2}{4C_s}\big(G^{n-\sigma_n}\big)^2
+ \|f(t_{n-\sigma_n})\|\,\|u_h^{n-\sigma_n}\|.
\end{equation}
Applying Lemma \ref{lem:discrete_gronwall}, we further have
\begin{equation}\label{eq:stable3}
\|u_h^n\|
\le
\|u_h^0\|
+\max_{1\le k\le n}\sum_{j=1}^k \mathbb{P}_{k-j}^{(k)}
\left(
\|f(t_{j-\sigma_j})\|
+\frac{C_4}{2\sqrt{C_s}}G^{j-\sigma_j}
\right)
+\max_{1\le j\le n}\frac{C_4}{2\sqrt{C_s}}G^{j-\sigma_j}.
\end{equation}
Finally, by Lemma \ref{lem:Pjn_t0}, we obtain \eqref{eq:stable}.
\end{proof}

Next, we derive error estimates for the numerical solution $u_h^n$. Define the discrete-in-time norm
\begin{equation*}
\|\cdot\|_{\hat{l}^{\infty}(L^2)}:=\max_{0\le n\le N}\|\cdot\|.
\end{equation*}
\begin{lemma}\label{lem:truncate_est}
Assume that the solution $u$ of the original problem \eqref{eq:original_eq}--\eqref{eq:original_eq3} satisfies Assumption \ref{assumption_a} and that $u(t)\in H^s(\Omega^-\cup\Omega^+)$, $s\ge 2$, for all $t\in[0,T]$. Let
\begin{equation*}
\alpha_n\in\big[\min_{t\in[t_{n-1},t_n]}\alpha(t),\ \max_{t\in[t_{n-1},t_n]}\alpha(t)\big],
\end{equation*}
with $\theta_n=\alpha_n/2$ and $\alpha_n^*=\alpha(t_{n-\theta_n})$.
Then there exist positive constants $C_8$, $C_9$ and $C_{10}$, independent of $\tau$ and $n$, such that
\begin{equation}\label{eq:Rn_est}
\|R^n\|\le C_8\, t_{n-\sigma_n}^{-\alpha_n^*}\,N^{-\min\{2,r\delta\}},
\end{equation}
\begin{equation}\label{eq:En_est}
\| E^n\|\le
\left\{
\begin{array}{ll}
C_9 t_{n-\sigma_n}^{-\alpha_n^*}N^{-2}, & r\ge {2}/{(\delta+\alpha_n^*)},\\[4pt]
C_9 t_{n-\sigma_n}^{-\alpha_n^*}N^{-r(\delta+\alpha_n^*)}
\le C_9 t_{n-\sigma_n}^{-\alpha_n^*}N^{-r\delta}, & 1\le r<{2}/{(\delta+\alpha_n^*)},
\end{array}
\right.
\end{equation}
and
\begin{equation}\label{eq:R3n_est0}
\|(D_{\tau}^{\alpha_n^*}(u-\Pi_hu))^{n-\theta_n}\|
\le \frac{1}{\delta}C_{10}\,t_{n-\sigma_n}^{-\alpha_n^*}h^\mu,
\end{equation}
where $\Pi_h$ is the elliptic projection operator defined in \eqref{eq:elliptic_project}, and $\mu=\min\{k+1,s\}$.
\end{lemma}
\begin{proof}For details, we refer to Lemma 11 in \cite{huang2026determining}.
\end{proof}

\begin{theorem}\label{thm:full_error}
Suppose that the solution $u$ of the original problem \eqref{eq:original_eq}--\eqref{eq:original_eq3} satisfies Assumption \ref{assumption_a} and that $u(t)\in H^s(\Omega^-\cup\Omega^+)$, $s\ge 2$, for all $t\in[0,T]$, and let $u_h^n$ be the solution of the discrete scheme \eqref{eq_whole_discrete}. Assume that Assumption \ref{assumption_a} and \ref{assumption_b} holds, and that $\alpha_n$ satisfies \eqref{sigma_cond}. Let $\sigma_n=\alpha_n/2$ and $\alpha_n^*=\alpha(t_{n-\sigma_n})$. 
Then
\begin{equation}\label{eq:full_error}
\|u-u_h\|_{\hat l^\infty(L^2)}
\le
C_e\left(N^{-\min\{2,r\delta\}}+h^\mu\right),
\end{equation}
where $\mu=\min\{k+1,s\}$.
\end{theorem}

\begin{proof}
Let
\begin{equation*}
\eta_h:=u-\Pi_hu,\qquad \xi_h:=\Pi_hu-u_h,\qquad e_h:=u-u_h=\eta_h+\xi_h .
\end{equation*}
Subtracting \eqref{eq_whole_discrete} from \eqref{eq_whole_discrete0}, for any $w_h\in\mathbb{V}_h^k$, we obtain
\begin{equation}\label{eq:error}
((D_{\tau}^{\alpha_n^*}e_h)^{n-\sigma_n},w_h)+\mathbb{B}_h(e_h^{n-\sigma_n},w_h)
=(R^n+E^n,w_h),\qquad n\ge1.
\end{equation}
Moreover, $e_h^0=\eta_h^0$ and $\xi_h^0=0$.
Taking $w_h=\xi_h^{n-\sigma_n}$, using $e_h=\eta_h+\xi_h$, and the elliptic projection property
\begin{equation*}
\mathbb{B}_h(\eta_h^{n-\sigma_n},\xi_h^{n-\sigma_n})=0,
\end{equation*}
we get
\begin{equation*}
((D_{\tau}^{\alpha_n^*}\xi_h)^{n-\sigma_n},\xi_h^{n-\sigma_n})
+\mathbb{B}_h(\xi_h^{n-\sigma_n},\xi_h^{n-\sigma_n})
=
\big(R^n+E^n-(D_{\tau}^{\alpha_n^*}\eta_h)^{n-\sigma_n},\xi_h^{n-\sigma_n}\big).
\end{equation*}
Applying the positivity of the bilinear $\mathbb{B}_h(\cdot,\cdot)$ and repeating the same argument as in Theorem \ref{thm:stable}, we have
\begin{equation}\label{eq:xi-est-raw}
\|\xi_h^n\|
\le
\|\xi_h^0\|
+\max_{1\le k\le n}\sum_{j=1}^k\mathbb{P}_{k-j}^{(k)}
\|R^j+E^j-(D_{\tau}^{\alpha_j^*}\eta_h)^{j-\sigma_j}\|.
\end{equation}
Since $\sigma_j=\alpha_n/2\in[0,1/2]$,
\begin{equation*}
t_{j-\sigma_j}=\sigma_j t_{j-1}+(1-\sigma_j)t_j\ge (1-\sigma_j)t_j\ge \frac12 t_j,
\end{equation*}
thus
\begin{equation*}
t_{j-\sigma_j}^{-\alpha_j^*}\le 2^{\alpha_j^*}t_j^{-\alpha_j^*}\le 2t_j^{-\alpha_j^*}.
\end{equation*}
By Lemma \ref{lem:truncate_est},
\begin{equation*}
\|R^j\|\le 2C_8 t_j^{-\alpha_j^*}N^{-\min\{2,r\delta\}},
\qquad
\| E^j\|\le 2C_9  t_j^{-\alpha_j^*}N^{-\min\{2,r\delta\}},
\end{equation*}
and
\begin{equation*}
\|(D_{\tau}^{\alpha_j^*}\eta_h)^{j-\sigma_j}\|
\le \frac{2}{\delta}C_{10}  t_j^{-\alpha_j^*}h^\mu .
\end{equation*}
Let $C_*:=\max\{2C_8+2C_9,2C_{10}/\delta\}$. Then
\begin{equation*}
\|R^j+E^j-(D_{\tau}^{\alpha_j^*}\eta_h)^{j-\sigma_j}\| 
\le 
C_*\, t_j^{-\alpha_j^*}\big(N^{-\min\{2,r\delta\}}+h^\mu\big).
\end{equation*}
Substituting into \eqref{eq:xi-est-raw} and using Lemma \ref{lem:Pjn_t0} and $\xi_h^0=0$, we obtain
\begin{equation*}
\|\xi_h^n\|
\le
C_{11}\big(N^{-\min\{2,r\delta\}}+h^\mu\big),\qquad 1\le n\le N,
\end{equation*}
where $C_{11}$ is independent of $h,\tau,n$.
On the other hand, the projection estimate \eqref{eq:projection_error_L2} yields
\begin{equation*}
\|\eta_h^n\|=\|u^n-\Pi_hu^n\|
\le C h^\mu \|u^n\|_{H^s(\Omega^-\cup\Omega^+)}
\le C h^\mu .
\end{equation*}
Therefore, triangle inequality satisfies
\begin{equation*}
\|u^n-u_h^n\|
\le \|\eta_h^n\|+\|\xi_h^n\|
\le C_e\big(N^{-\min\{2,r\delta\}}+h^\mu\big),\qquad 1\le n\le N.
\end{equation*}
Taking the maximum over $0\le n\le N$ gives
\begin{equation*}
\|u-u_h\|_{\hat l^\infty(L^2)}
\le
C_e\big(N^{-\min\{2,r\delta\}}+h^\mu\big).
\end{equation*}
This completes the proof.
\end{proof}

 \section{Numerical examples}
 
In this section, we present three numerical examples to validate the theoretical results and assess the effectiveness of the proposed method. We examine the effects of initial singularity, complex interface geometry, low spatial regularity, and the superconvergent points $(t_{n-\sigma_n})$ on convergence behavior. For complex interfaces $\Gamma$, interface-fitted meshes are generated using MATLAB pdetool, and volume integrals on curved elements are computed as in Section 2.2 of \cite{huang2020high}. Exact solutions are available for all examples.

 Let $N_e$ be total number of elements.
In each example, we measure errors using the following notations $e_h:=u-u_h$ and 
\begin{equation*}\|e_h\|_{\hat{l}^{\infty}(L_2)}:=\max_{1\leq n\leq N}\|e_h^n\|_{L^2(\Omega)}, \qquad\text{order}:=\frac{\|e_h\|_{\hat{l}^{\infty}(L_2)}}{\|e_{h/2}\|_{\hat{l}^{\infty}(L_2)}}.
\end{equation*}

\begin{example}\label{examp1}
\textbf{(Initial singularity in time)}  
Let $\Omega=(-1,1)^2$ and $[0,T]=[0,1]$.  
The interface is
\begin{equation*}
\Gamma=\{(x,y):x^2+y^2=1/4\}.
\end{equation*}
The diffusion coefficient is defined by
\begin{equation*}
\beta(x,y)=
\begin{cases}
2+x+y, &  (x,y)\in \Omega^-,\\
10, &  (x,y)\in \Omega^+.
\end{cases}
\end{equation*}
The exact solution is chosen as
\begin{equation*}\label{eq:sol_circle_exa2}
u(x,y,t)=
\begin{cases}
(1+t^{\alpha(0)})(x^2+y^2), &  (x,y)\in \Omega^-,\\
(1+t^{\alpha(0)})(x^2+y^2)^{3/2}, &  (x,y)\in \Omega^+.
\end{cases}
\end{equation*}
We consider two variable-order functions $\alpha(t)$: one is monotonically increasing, 
\begin{equation*}
\alpha_1(t)=0.9+\big(\alpha(0)-0.9\big)\left(1-t-\frac{\sin\!\big(2\pi(1-t)\big)}{2\pi}\right),
\end{equation*}
and the other is monotonically decreasing,
\begin{equation*}
\alpha_2(t)=\alpha(0)\exp(-t).
\end{equation*}
The source term $f$, initial value $u_0$, and interface/boundary data $\psi$, $\phi$, and $g$ are derived from the exact solution.
\end{example}

Since the exact solution $u\in C^{\infty}(\Omega^+\cup\Omega^-)$ for all $t\in (0,1)$ and has $\alpha(0)$-order initial singularity, 
Theorem \ref{thm:full_error} yields
\begin{equation}\label{eq:examp1_est}
\|u-u_h\|_{\hat{l}^{\infty}(L_2)}\leq C_e\left(N^{-\min\{2,r\alpha(0)\}}+h^{k+1}\right).
\end{equation}

First, we consider the case of the monotonically increasing order function $\alpha_1(t)$. We begin by examining the effect of the initial singularity on the temporal convergence rate. In Table \ref{tab:examp11}, we choose $\alpha_n=\alpha_1\!\left(t_{n-1/2}\right)$, 
and take the superconvergent point as $t_{n-\alpha_n/2}$. Since $\alpha_1(t)$ is monotonically increasing, $\alpha_n$ satisfies condition \eqref{sigma_cond}. Table \ref{tab:examp11} shows that the convergence order in time is $\min\{2,\,r\alpha_1(0)\}$.

Second, we examine the influence of the superconvergent point $t_{n-\alpha_n/2}$ on the temporal convergence rate. In Table \ref{tab:examp12}, we choose
\begin{equation*}
\alpha_n=\alpha_1(t_{n-0.6}),\qquad
\alpha_n=\alpha_1(t_{n-0.8}),\qquad
\alpha_n=\alpha_1(t_{n-0.9}),
\end{equation*}
all of which satisfy condition \eqref{sigma_cond}. Table \ref{tab:examp12} shows that different choices of superconvergent points do not affect the convergence order with respect to the time mesh; the order depends only on the initial singularity index $\alpha_1(0)$. Moreover, Table \ref{tab:examp12} indicates that there are many admissible superconvergent points at each time step, whereas only one such point was used in \cite{du2020temporal}.

Third, we examine the effects of the spatial polynomial degree and different mesh parameters on the convergence rate. In Table \ref{tab:examp13}, we set $N=500$, choose $k=1,2$, and use three different meshes. Table \ref{tab:examp13} shows that the spatial convergence order is $k+1$, which is consistent with the estimate in \eqref{eq:examp1_est}.

Finally, we consider the case of the monotonically decreasing order function $\alpha_2(t)$. Tables \ref{tab:examp1_1}--\ref{tab:examp1_2} show that the numerical results are consistent with those obtained for $\alpha_1(t)$.

\begin{table}[h!] 
\centering    
\caption{Errors and convergence orders for different temporal mesh and $\alpha_1(0)$ with $\alpha_n = \alpha_1(t_{n-1/2})$ and $N_e=21248$ in Example \ref{examp1} }
\label{tab:examp11}
\begin{tabular}{ccccccccccccc}
\toprule
&    & $\alpha_1(0)=0.4$ & & $\alpha_1(0)=0.6$ & &$\alpha_1(0)=0.8$ & \\ 
\midrule
$r$  & $N$    & $\|e_h\|_{l^{\infty}(L_{2})}$   & order  & $\|e_h\|_{l^{\infty}(L_{2})} $&  order & $\|e_h\|_{l^{\infty}(L_{2})} $ & order \\ 
\midrule
1 &8 & 1.7583e-01 &          & 1.2391e-01 &          &1.0452e-01  &      
&\\    
&16  & 1.3289e-01 & 0.40   & 8.1559e-02 & 0.60   &5.9818e-02  &0.80  \\  
&32  & 1.0060e-01 & 0.40   & 5.3682e-02 & 0.60   &3.4248e-02  &0.80  \\  
\midrule
2 &8 & 7.6155e-02 &          & 3.5365e-02 &      &1.0894e-02  &      
\\    
&16  & 4.3619e-02 & 0.80   & 1.5410e-02 & 1.19   &3.5936e-03  &1.60  \\  
&32  & 2.5021e-02 & 0.80   & 6.7152e-03 & 1.19   &1.1851e-03  &1.60  \\  
\midrule
3 &8 & 3.3336e-02 &        &1.1801e-02  &        &7.8066e-03  &      
\\    
&16  & 1.4510e-02 & 1.20   &3.2128e-03  &1.87    &1.9525e-03  &1.99  \\  
&32  & 6.3157e-03 & 1.20   &9.2237e-04  &1.80    &4.8891e-04  &1.99  \\  
\midrule
4 &8 &2.2030e-02  &        &2.1193e-02  &        &1.3758e-02  &      \\    
&16  &5.6808e-03  &1.95    &5.1912e-03  &2.02    &3.4561e-03  &1.99  \\  
&32  &1.8737e-03  &1.60    &1.3035e-03  &1.99    &8.6798e-04  &1.99  \\  
\bottomrule    
\end{tabular}
\end{table}

\begin{table}[h!] 
\centering   
\caption{ Errors and convergence orders for superconvergent points $t_{n-\alpha_n/2}$ with graded mesh $r=1$, $k=2$ and $N_e=21248$  in Example \ref{examp1} }
\label{tab:examp12}
\begin{tabular}{ccccccccccccc}
\toprule
& &$\alpha_n=\alpha_1(t_{n-0.6})$&  &$\alpha_n=\alpha_1(t_{n-0.8})$ & &$\alpha_n=\alpha_1(t_{n-0.9})$\\ 
\midrule
$\alpha_1(0)$&$N$ & $\|e_h\|_{L^{\infty}(L_2)}$&order&$\|e_h\|_{L^{\infty}(L_2)}$& order& $\|e_h\|_{L^{\infty}(L_2)}$ &order\\ 
\midrule
$0.2$ &8  &1.6361e-01  &     &1.6357e-01  &       &1.0656e-01             &    \\      
      &16 &1.4199e-01  &0.20 &1.4185e-01  &0.20   &1.4182e-01 &0.20\\   
      &32 &1.2352e-01  &0.20 &1.2343e-01  &0.20   &1.2340e=01 &0.20\\  
\midrule
$0.4$ &8  &1.7563e-01  &     &1.7544e-01  &       &1.7541e-01             &    \\      
      &16 &1.3287e-01  &0.40 &1.3285e-01  &0.40   &1.3281e-01 &0.40\\      
      &32 &1.0060e-01  &0.40 &1.0059e-01  &0.40   &1.0053e-01 &0.40\\      
\hline
$0.6$ &8  &1.2384e-01  &     &1.2379e-01  &       &1.2378e-01             &    \\         
      &16 &8.1550e-02  &0.60 &8.1545e-02  &0.60   &8.1543e-02    &0.60\\      
      &32 &5.3678e-02  &0.60 &5.3678e-02  &0.60   &5.3672e-02 &0.60\\      
\bottomrule  
\end{tabular}
\end{table}

\begin{table}[h!] 
\centering     
\caption{Errors and convergence orders for different spatial mesh and $\alpha(0)$ with $\alpha_n = \alpha_1(t_{n-1/2})$ and $N=500$ in Example \ref{examp1}}  
\label{tab:examp13}
\begin{tabular}{ccccccccccccc}
\toprule
&    & $\alpha_1(0)=0.4$ & & $\alpha_1(0)=0.6$ & &$\alpha_1(0)=0.8$ & \\ 
\midrule
$k$ & $N_e$& $\|e_h\|_{l^{\infty}(L_{2})}$   & order  & $\|e_h\|_{l^{\infty}(L_{2})} $& order & $\|e_h\|_{l^{\infty}(L_{2})} $ & order \\ 
\midrule
1 &332   & 5.9011e-02 &          & 5.7416e-02 &           &5.6026e-02&             \\  
  &1328  & 1.2883e-02 & 2.19   & 1.2546e-02 & 2.19    &1.2459e-02& 2.16  \\  
  &5312   & 3.4912e-03 & 1.88   & 3.3726e-03 & 1.89    &3.3556e-03& 1.89  \\  
\midrule
2 &332   & 1.0814e-03 &          & 1.0529e-03 &           &1.02682e-03&            \\    
  &1328   & 1.4015e-04 & 2.94   & 1.3619e-04 & 2.95    &1.32692e-04
& 2.95  \\    
  &5312   & 1.7632e-05 & 2.99   & 1.7045e-05 & 2.99    &1.67064e-05
& 2.98  \\    
\bottomrule
\end{tabular}
\end{table}   

\begin{table}[h!] 
\centering   
\caption{Errors and convergence orders for different temporal mesh and $\alpha(0)$ with $\alpha_n = \alpha_2(t_{n-1/2})$, $k=2$ and  $N_e=21248$ in Example \ref{examp1} }
\label{tab:examp1_1}
\begin{tabular}{ccccccccccccc}
\toprule
&    & $\alpha_2(0)=0.4$ & & $\alpha_2(0)=0.6$ & &$\alpha_2(0)=0.8$ & \\ 
\midrule
$r$  & $N$    & $\|e_h\|_{l^{\infty}(L_{2})}$   & order  & $\|e_h\|_{l^{\infty}(L_{2})} $&  order & $\|e_h\|_{l^{\infty}(L_{2})} $ & order \\ 
\midrule
1 &8  &1.6334e-01 &      &1.1472e-01 &       &5.2833e-02  &     &\\    
  &16 &1.2847e-01 &0.35  &7.8758e-02 &0.54   &3.1649e-02  &0.74  \\  
  &32 &1.2846e-01 &0.37  &5.3010e-02 &0.57   &1.8567e-02  &0.77  \\  
\midrule
2 &8  &7.5875e-02 &      &3.5326e-02 &       &1.0779e-02  &    
  \\    
  &16 &4.3885e-02 &0.79  &1.5492e-02 &1.19   &3.5840e-03  &1.59  \\  
  &32 &2.5249e-02 &0.80  &6.7561e-03 &1.20   &1.1843e-03  &1.60  \\  
\midrule
3 &8  &3.3297e-02 &      &1.1086e-02 &       &5.3265e-03  &      \\    
  &16 &1.4508e-02 &1.20  &3.2092e-03 &1.79   &1.3815e-03  &1.95  \\  
  &32 &6.3156e-03 &1.20  &9.2224e-04 &1.80   &3.4911e-04  &1.98  \\  
\midrule
4 &8  &1.7182e-02 &       &1.2042e-02 &      &7.8726e-03  &      \\    
  &16 &5.6801e-03 &1.59   &3.1917e-03 &1.92  &2.0407e-03  &1.99  \\  
  &32 &1.8737e-03 &1.60   &8.1816e-04 &1.96  &5.1833e-04  &1.99  \\  
\bottomrule    
\end{tabular}
\end{table}

\begin{table}[h!] 
\centering  
\caption{ Errors and convergence orders for superconvergent points $t_{n-\alpha_n/2}$ with graded mesh $r=2$, $k=2$ and  $N_e=21248$ in Example \ref{examp1} }
\label{tab:examp1_2}
\begin{tabular}{ccccccccccccc}
\toprule
& &$\alpha_n=\alpha_2(t_{n-0.6})$&  &$\alpha_n=\alpha_2(t_{n-0.8})$ & &$\alpha_n=\alpha_2(t_{n-\alpha_n/2})$\\ 
\midrule
$\alpha_2(0)$&$N$ & $\|e_h\|_{L^{\infty}(L_2)}$&order&$\|e_h\|_{L^{\infty}(L_2)}$& order& $\|e_h\|_{L^{\infty}(L_2)}$ &order\\ 
\midrule
$0.2$ &8  &1.0674e-01 &     &1.0670e-01 &      &1.0655e-01 &\\            
      &16 &8.1201e-02 &0.39 &8.1197e-02 &0.39  &8.1194e-02 &0.39\\         
      &32 &6.1557e-02 &0.40 &6.1550e-02 &0.40  &6.1546e-02 &0.40\\        
\midrule
$0.4$ &8  &7.5401e-02 &     &7.5428e-02 &      &7.5441e-02 &\\            
      &16 &4.3489e-02 &0.79 &4.3501e-02 &0.79  &4.3516e-02 &0.79\\            
      &32 &2.4991e-02 &0.80 &2.5001e-02 &0.80  &2.5007e-02 &0.80\\            
\hline
$0.6$ &8  &3.5075e-02 &     &3.5017e-02 &      &3.5004e-02 &\\            
      &16 &1.5378e-02 &1.19 &1.5372e-02 &1.19  &1.5370e-02 &1.18\\            
      &32 &6.7735e-03 &1.18 &6.7730e-03 &1.18  &6.7726e-03 &1.18\\            
\bottomrule  
\end{tabular}
\end{table}

\begin{example}\label{examp2}
\textbf{(Low regularity solution in space)}  
Let $\Omega=(-1,1)^2$ and $[0,T]=[0,1]$.  
The interface is
\begin{equation*}
\Gamma=\{(x,y):x^2+y^2=(\pi/6.28)^2\}.
\end{equation*}
The diffusion coefficient is defined by
\begin{equation*}
\beta(x,y)=
\begin{cases}
2+x+y, & (x,y)\in \Omega^-,\\
2, & (x,y)\in\Omega^+.
\end{cases}
\end{equation*}
We consider two variable-order functions $\alpha(t)$: one is monotonically increasing, 
\begin{equation*}
\alpha_1(t)=0.9+\big(0.6-0.9\big)\left(1-t-\frac{\sin\!\big(2\pi(1-t)\big)}{2\pi}\right),
\end{equation*}
and the other is non-monotonic,
\begin{equation*}
\alpha_2(t)=0.5+0.2\sin(2\pi t).\end{equation*}
The exact solution is chosen as
\begin{equation*} 
u(x,y,t)=
\begin{cases}
(1+t^{\delta})(x^2+y^2)^{1.1}, & (x,y)\in \Omega^-,\\
(1+t^{\delta})((x-1)^2+(y-1)^2)^{1.1}, &  (x,y)\in \Omega^+.
\end{cases}
\end{equation*}
The source term $f$, initial value $u_0$, and interface/boundary data $\psi$, $\phi$, and $g$ are derived from the exact solution.
\end{example}

Since the exact solution $u\in H^{3.2-\varepsilon}(\Omega^+\cup\Omega^-)$ for any $\varepsilon>0$ and all $t\in (0,1)$ and has $\delta$-order initial singularity, 
Theorem \ref{thm:full_error} yields
\begin{equation}\label{eq:examp2_est}
\|u-u_h\|_{\hat{l}^{\infty}(L_2)}\leq C_e\left(N^{-\min\{2,r\delta\}}+h^{\min\{k+1,3.2-\varepsilon\}}\right).
\end{equation}

In this example, we check the effects on the error of the initial singularity of the solution, the spatial singularity of the solution, and the non-monotonic variable-order function $\alpha(t)$.

First, since the solution $u$ satisfies $|u|\le C(1+t^{\delta})$, Table \ref{tab:examp21} shows that the temporal convergence order of the approximate solution $u_h$ is $\min\{2,r\delta\}$. 
Moreover, Table \ref{tab:examp12} indicates that there are many admissible superconvergent points at each time step.
Furthermore, because $\alpha_n=\alpha(t_{n-0.2})\in[0.6,0.9]$, we have $\alpha_n>\alpha(t_{n-\alpha_n/2})$, so the choice of $\alpha_n$ violates the first inequality in \eqref{sigma_cond}. Nevertheless, the numerical solution remains convergent, and the observed temporal order is still $\min\{2,r\delta\}$. Therefore, the condition $\alpha_n\leq\alpha(t_{n-\alpha_n/2})$ in \eqref{sigma_cond} is a sufficient condition, but not a necessary one.

Second, Table \ref{tab:examp22} shows that the spatial convergence order is $\min\{k+1,3.2\}$, which is consistent with the estimate in \eqref{eq:examp2_est}.

Finally, we provide a non-monotonic variable-order function $\alpha_2(t)$. Table show According to Table \ref{tab:examp23}, if $\alpha_n$ is in the range of the variable-order function $\alpha_2(t)$ over $[t_{n-1},t_n]$, then the temporal convergence order is $\min\{2,r\delta\}$.

\begin{table}[h!] 
\centering   
\caption{ Errors and convergence orders for superconvergent points $t_{n-\alpha_n/2}$ with $\delta=0.6$, $k=2$ and $N_e=5216$  in Example \ref{examp2} }
\label{tab:examp21}
\begin{tabular}{ccccccccccccc}
\toprule
 &&$\alpha_n=\alpha_1(t_{n-0.2})$&  &$\alpha_n=\alpha_1(t_{n-0.6})$ & &$\alpha_n=\alpha_1(t_{n-0.8})$&\\ 
\midrule
$r$&$N$ & $\|e_h\|_{L^{\infty}(L_2)}$&order&$\|e_h\|_{L^{\infty}(L_2)}$& order& $\|e_h\|_{L^{\infty}(L_2)}$ &order\\ 
\midrule
1& 16 &1.6090e-01 &&  1.6091e-01 &&  1.6100e-01&\\
 &32 & 9.8548e-02 &0.71&  9.8549e-02 &0.71&  9.8556e-02&0.71\\
&  64& 5.9531e-02 &0.73&  5.9531e-02  &0.73& 5.9531e-02&0.73\\  
 \midrule
2& 16 &2.1185e-02 &&   2.1185e-02 &&   2.1185e-02&\\
 &32&   7.8599e-03 &1.43&   7.8599e-03 &1.43&   7.8599e-03&1.43\\
  &64&  3.2894e-03  &1.26&  3.2894e-03  &1.26&  3.2894e-03&1.26\\\hline
3&16& 7.2056e-03&&   7.1274e-03  && 7.1108e-03&\\
 &32&  1.8067e-03 &1.99&  1.7891e-03  &1.99& 1.7843e-03&1.99\\
 &64&  4.5230e-04 &1.99&  4.4859e-04  &1.99& 4.4737e-04&1.99\\\hline
 4& 16 &  1.2526e-02  && 1.2555e-02  && 1.2548e-02&\\
  &32& 3.1529e-03  &1.99& 3.1509e-03 &1.99&  3.1484e-03&1.99\\
 &64&  7.9063e-04  &1.99& 7.8982e-04  &1.99& 7.8929e-04&1.99\\  
\bottomrule  
\end{tabular}
\end{table}

\begin{table}[h!] 
\centering     
\caption{Errors and convergence orders for different spatial mesh and  with $\delta=2.4$ and $N=2000$ in Example \ref{examp2}}  
\label{tab:examp22}
\begin{tabular}{ccccccccccccc}
\toprule
&    & $\alpha_n=\alpha_1(t_{n-0.8})$ & & $\alpha_n=\alpha_1(t_{n-0.6})$ & &$\alpha_n=\alpha_1(t_{n-0.2})$ & \\ 
\midrule
$k$ & $N_e$& $\|e_h\|_{l^{\infty}(L_{2})}$   & order  & $\|e_h\|_{l^{\infty}(L_{2})} $& order & $\|e_h\|_{l^{\infty}(L_{2})} $ & order \\ 
\midrule
1 &326   & 1.4861e-02 & &  1.4861e-02 & &  1.4861e-02&\\
 &1304 & 3.9572e-03  & 1.91& 3.9572e-03 & 1.91&  3.9572e-03&1.91 \\
 &5216 &  1.0240e-03  & 1.95& 1.0240e-03  & 1.95& 1.0240e-03&1.95 \\  
\midrule
2 &326   & 1.0057e-04 & &   1.0057e-04 & &   1.0057e-04&\\
 &1304 &   1.3377e-05 & 2.91&   1.3377e-05   &2.91 & 1.3377e-05&2.91\\
  & 5216&  1.8645e-06 & 2.84&   1.8647e-06  &2.84 &  1.8649e-06& 2.84\\    
\bottomrule
\end{tabular}
\end{table}   

\begin{table}[h!] 
\centering   
\caption{ Errors and convergence orders for superconvergent points $t_{n-\alpha_n/2}$ with $\delta=0.5$, $k=2$ and $N_e=5216$  in Example \ref{examp2} }
\label{tab:examp23}
\begin{tabular}{ccccccccccccc}
\toprule
 &&$\alpha_n=\alpha_2(t_{n-0.2})$&  &$\alpha_n=\alpha_2(t_{n-0.5})$ & &$\alpha_n=\alpha_2(t_{n-0.8})$&\\ 
\midrule
$r$&$N$ & $\|e_h\|_{L^{\infty}(L_2)}$&order&$\|e_h\|_{L^{\infty}(L_2)}$& order& $\|e_h\|_{L^{\infty}(L_2)}$ &order\\ 
\midrule
1& 16 &  2.5313e-01 &&  2.3815e-01 &&  2.2321e-01&\\
  &32& 1.5931e-01 &0.66&  1.5414e-01 &0.62&  1.4901e-01&0.58\\
 &64&  1.0259e-01 &0.63&  1.0083e-01 &0.61&  9.9080e-02&0.59\\
 \midrule
2& 16 & 4.3168e-02 &&  4.2973e-02 &&  4.2779e-02&\\
 &32&  1.8206e-02 &1.24&  1.8188e-02  &1.24& 1.8169e-02&1.23\\
&64&   8.0609e-03  &1.17& 8.0596e-03  &1.17& 8.0583e-03&1.17\\\hline
3&16&9.5804e-03 &&  9.5618e-03 &&  9.5431e-03& \\
   &32&2.8078e-03  &1.77& 2.8079e-03  &1.77& 2.8079e-03& 1.76\\
   &64&1.0656e-03  &1.40& 1.0656e-03  &1.40& 1.0656e-03 &1.39\\ \hline
 4& 16 &1.1290e-02 &&  1.1440e-02 &&  1.1502e-02&\\
  &32& 2.8896e-03  &1.97& 2.8964e-03 &1.98&  2.9129e-03&1.98\\
  &64& 7.3009e-04 &1.98&  7.3139e-04 &1.99&  7.3322e-04  &1.99 \\
\bottomrule  
\end{tabular}
\end{table}

\begin{example}\label{examp3}
\textbf{Interface with complex geometry}

In this example, we consider an elliptic interface problem \cite{huang2020high,mu2013weak} with a flower pedal shape interface that consists both concave and convex curved segments. The computational domain is $\Omega=(-1,1)^2$. The interface $\Gamma$ is parameterized with polar coordinates $(r,\theta)$ as,
\begin{equation*}r(\theta)=\frac{1}{2}+\frac{\sin (5\theta)}{7},\quad \theta\in[0,2\pi].\end{equation*}
The diffusion coefficient is defined by
\begin{equation*}
\beta(x,y)=
\begin{cases}
1, & (x,y)\in \Omega^-,\\
10, & (x,y)\in\Omega^+.
\end{cases}
\end{equation*}
Consider monotonically decreasing variable-order functions $\alpha(t)$
\begin{equation*}
\alpha(t)=0.8\exp(-t).
\end{equation*}
The analytical solution is given as,
\begin{equation*}
u_{ex}(r,\theta)=\left\{\begin{array}{ll}
(1+t^{2.4})\exp(r^2),& (x,y)\in\Omega^-,\\
(1+t^{2.4})(0.1r^4-0.01\ln(2r)),&(x,y)\in\Omega^+.
\end{array}\right.
\end{equation*}
\end{example}

 From Tables \ref{tab:examp31}--\ref{tab:examp32}, we observe that, even for a complex  interface $\Gamma$, the error in the approximate solution $u_h$ remains $O(N^{-2}+h^{k+1})$.
Moreover, choosing $\alpha_n=\alpha(t_{n-0.4})>\alpha(t_{n-\alpha_n/2})$  and $\alpha_n=\alpha(t_{n-0.8})>\alpha(t_{n-\alpha_n/2})$ does not affect the observed convergence order.
\begin{table}[h!] 
\centering   
\caption{ Errors and convergence orders for $k=2$ and $N_e=5344$  in Example \ref{examp3} }
\label{tab:examp31}
\begin{tabular}{cccccccccccc}
\toprule
 &$\alpha_n=\alpha(t_{n-0.2})$&  &$\alpha_n=\alpha(t_{n-0.4})$ & &$\alpha_n=\alpha(t_{n-0.8})$&\\ 
\midrule
$N$ & $\|e_h\|_{L^{\infty}(L_2)}$&order&$\|e_h\|_{L^{\infty}(L_2)}$& order& $\|e_h\|_{L^{\infty}(L_2)}$ &order\\ 
\midrule
 16 &9.5422e-04 &&  9.6450e-04 &&  9.8515e-04&\\
 32 & 2.3952e-04 &1.99&  2.4090e-04 &2.00&  2.4367e-04&2.01\\
 64&  6.0039e-05  &2.00& 6.0227e-05 &2.00&  6.0605e-05&2.00\\  
\bottomrule  
\end{tabular}
\end{table}

\begin{table}[h!] 
\centering     
\caption{Errors and convergence orders for different spatial mesh and  with $\delta=2.4$ and $N=1000$ in Example \ref{examp3}}  
\label{tab:examp32}
\begin{tabular}{ccccccccccccc}
\toprule
&    & $\alpha_n=\alpha(t_{n-0.8})$ & & $\alpha_n=\alpha(t_{n-0.5})$ & &$\alpha_n=\alpha(t_{n-0.1})$ & \\ 
\midrule
$k$ & $N_e$& $\|e_h\|_{l^{\infty}(L_{2})}$   & order  & $\|e_h\|_{l^{\infty}(L_{2})} $& order & $\|e_h\|_{l^{\infty}(L_{2})} $ & order \\ 
\midrule
1 &334  & 8.8956e-03 & &  8.8956e-03 & &  8.8956e-03&\\
 & 1336&  2.6368e-03  &1.75 & 2.6368e-03 &1.75 &  2.6368e-03&1.75\\
 & 5344&  6.8302e-04 &1.95 &  6.8302e-04  &1.95 & 6.8302e-04 & 1.95\\  
\midrule
2 &334   & 2.9523e-04 & &  2.9523e-04  & & 2.9523e-04&\\
  & 1336& 4.5359e-05 &2.70 &  4.5359e-05 & 2.70&  4.5359e-05&2.70\\
  &5344 & 5.9961e-06  &2.92 & 5.9961e-06 &2.92 &  5.9961e-06 &2.92\\    
\bottomrule
\end{tabular}
\end{table}   

\section{Conclusion}\label{sec_5}

This work has study a variable-order time-fractional subdiffusion interface problem with discontinuous diffusion coefficients across a curved interface. To capture nonlocal memory, spatial-temporal heterogeneity, and abrupt changes of material parameters at interfaces, we propose a fully discrete scheme that couples the $L2\!-\!1_\sigma$  approximation of the variable-order Caputo derivative on graded temporal meshes with a symmetric interior penalty finite element discretization in space on body-fitted triangulations. The graded time mesh is employed to alleviate the initial singularity, while the interior penalty formulation provides a flexible framework for treating coefficient jumps and geometrically complex interfaces.

Under suitable assumptions on the interface geometry and the regularity of the exact solution, we establish stability of the fully discrete method and derived optimal a priori error estimate in an appropriate discrete-in-time $L^2$ norm, achieving second-order temporal accuracy. The analysis shows that the discrete parameters associated with the evaluation points
$t_{n-\alpha_n/2}$--in particular, the choice of $\alpha_n$ relative to $\alpha(t_{n-\alpha_n/2})$, 
are central to the $L2\!-\!1_{\sigma}$ discretization and to the stability and error estimates. Theoretically, second-order accuracy and stability are usually proved under
\[
\alpha_n\in\left[\min_{t\in[t_{n-1},t_n]}\alpha(t),\ \max_{t\in[t_{n-1},t_n]}\alpha(t)\right]
\quad\text{and}\quad
\alpha_n\le \alpha\!\left(t_{n-\alpha_n/2}\right).
\]
Numerical tests indicate that the second inequality can be relaxed or even removed without losing convergence. Moreover, $\alpha_n$ can be chosen straightforwardly at each time step from a large family of  admissible values, so that  it is unnecessary to solve a nonlinear equation for a unique $\alpha_n$.
 These findings motivate a more practical parameter-selection strategy with multiple admissible choices at each time level.

 Extensive tests confirm the predicted temporal accuracy $\min{2,r\delta}$ (or the analogous rate dictated by the initial singularity), the spatial order $\min\{s,k+1\}$ for polynomials of degree $k$, and the robustness of the method with respect to different choices of superconvergent points and complex interface configurations.

Future work may include extending the present framework to other variable-order fractional models and developing fast algorithms for long-time simulation.

 \section*{CRediT authorship contribution statement} 
 
\textbf{Hongying Huang:} {Conceptualization of this study, Methodology, Writing - original draft, Writing - review \& editing}; 
\textbf{Chanchan Hao: }{Investigation, Methodology, Validation, Writing - original draft}; 
\textbf{Changmu Yu:} {Investigation, Validation, Software}; 
\textbf{Huili Zhang:} {Validation, Software, Funding acquisition, Supervision, Writing - review \& editing}.

\section*{Data availability}

No data was used for the research described in the article.

\section*{Declaration of competing interest}

This work does not have any conflicts of interest.

\section*{Acknowledgement}

The authors are grateful to the anonymous referees for their valuable comments
and suggestions, which helped to improve the article.

Huang's work was supported by the National Natural Science Foundation of China (Grant No. 11771398) and the Innovation Team Project of Regular Universities in Guangdong Province (2025KCXTD037). Zhang's work was supported by the Tertiary Education Scientific Research Project of Guangzhou Municipal
Education Bureau (2024312092).

 \bibliographystyle{plain}

\bibliography{subdiffusion_references}

@article{du2020temporal,
  author  = {Du, R. and Alikhanov, A. A. and Sun, Z.},
  title   = {Temporal second order difference schemes for the multi-dimensional variable-order time fractional sub-diffusion equations},
  journal = {Computers \& Mathematics with Applications},
  volume  = {79},
  number  = {10},
  pages   = {2952--2972},
  year    = {2020},
  doi     = {10.1016/j.camwa.2020.01.003}
}

@article{gu2023two,
  author  = {Gu, Q. and Chen, Y. and Zhou, J. and Huang, Y.},
  title   = {A two-grid virtual element method for nonlinear variable-order time-fractional diffusion equation on polygonal meshes},
  journal = {International Journal of Computer Mathematics},
  volume  = {100},
  number  = {11},
  pages   = {2124--2139},
  year    = {2023},
  doi     = {10.1080/00207160.2023.2263589}
}

@article{heydari2019computational,
  author  = {Heydari, M. H. and Avazzadeh, Z. and Yang, Y.},
  title   = {A computational method for solving variable-order fractional nonlinear diffusion-wave equation},
  journal = {Applied Mathematics and Computation},
  volume  = {352},
  pages   = {235--248},
  year    = {2019},
  doi     = {10.1016/j.amc.2019.01.075}
}

@article{heydari2020cardinal,
  author  = {Heydari, M. H. and Avazzadeh, Z. and Yang, Y. and Cattani, C.},
  title   = {A cardinal method to solve coupled nonlinear variable-order time fractional sine-{G}ordon equations},
  journal = {Computational and Applied Mathematics},
  volume  = {39},
  number  = {1},
  pages   = {2},
  year    = {2020},
  doi     = {10.1007/s40314-019-0936-z}
}

@article{huang2023alpha,
  author  = {Huang, C. and An, N. and Chen, H. and Yu, X.},
  title   = {$\alpha$-robust error analysis of two nonuniform schemes for subdiffusion equations with variable-order derivatives},
  journal = {Journal of Scientific Computing},
  volume  = {97},
  pages   = {43},
  year    = {2023},
  doi     = {10.1007/s10915-023-02357-5}
}

@article{liao2019discrete,
  author  = {Liao, H. and McLean, W. and Zhang, J.},
  title   = {A discrete {G}r\"{o}nwall inequality with applications to numerical schemes for subdiffusion problems},
  journal = {SIAM Journal on Numerical Analysis},
  volume  = {57},
  number  = {1},
  pages   = {218--237},
  year    = {2019},
  doi     = {10.1137/16M1175742}
}

@article{ma2023stabilizer,
  author  = {Ma, J. and Gao, F. and Du, N.},
  title   = {A stabilizer-free weak {G}alerkin finite element method to variable-order time fractional diffusion equation in multiple space dimensions},
  journal = {Numerical Methods for Partial Differential Equations},
  volume  = {39},
  number  = {3},
  pages   = {2096--2114},
  year    = {2023},
  doi     = {10.1002/num.22959}
}

@article{metzler2000random,
  author  = {Metzler, R. and Klafter, J.},
  title   = {The random walk's guide to anomalous diffusion: a fractional dynamics approach},
  journal = {Physics Reports},
  volume  = {339},
  number  = {1},
  pages   = {1--77},
  year    = {2000},
  doi     = {10.1016/S0370-1573(00)00070-3}
}

@article{metzler2000subdiffusive,
  author  = {Metzler, R. and Klafter, J.},
  title   = {Subdiffusive transport close to thermal equilibrium: {F}rom the {L}angevin equation to fractional diffusion},
  journal = {Physical Review E},
  volume  = {61},
  number  = {6},
  pages   = {6308--6311},
  year    = {2000},
  doi     = {10.1103/physreve.61.6308}
}

@ARTICLE{mu2013weak,
  AUTHOR =       {Mu, L. and Wang, J. P. and Wei, G. W. and Ye, X. and Zhao, S.},
  TITLE =        {Weak {G}alerkin methods for second order elliptic interface problems},
  JOURNAL =      {Journal of Computational Physics },
  YEAR =         {2013},
  volume =       {250},
  pages =        {106--125},
  doi = {10.1016/j.jcp.2013.04.042},
}

@book{podlubny1998fractional,
  author    = {Podlubny, I.},
  title     = {Fractional Differential Equations: An Introduction to Fractional Derivatives, Fractional Differential Equations, to Methods of Their Solution and Some of Their 
  Applications},
  series    = {Mathematics in Science \& Engineering},
  volume    = {198},
  publisher = {Academic Press},
  address   = {San Diego, CA},
  year      = {1998}
}

@article{umarov2009variable,
  author  = {Umarov, S. and Steinberg, S.},
  title   = {Variable order differential equations with piecewise constant order-function and diffusion with changing modes},
  journal = {Zeitschrift fur Analysis und ihre Anwendungen},
  volume  = {28},
  number  = {4},
  pages   = {431--450},
  year    = {2009},
  doi     = {10.4171/ZAA/1392}
}

@article{wang2019analysis,
  author  = {Wang, H. and Zheng, X.},
  title   = {Analysis and numerical solution of a nonlinear variable-order fractional differential equation},
  journal = {Advances in Computational Mathematics},
  volume  = {45},
  number  = {5},
  pages   = {2647--2675},
  year    = {2019},
  doi     = {10.1007/s10444-019-09690-0}
}

@article{wang2019wellposedness,
  author  = {Wang, H. and Zheng, X.},
  title   = {Wellposedness and regularity of the variable-order time-fractional diffusion equations},
  journal = {Journal of Mathematical Analysis and Applications},
  volume  = {475},
  number  = {2},
  pages   = {1778--1802},
  year    = {2019},
  doi     = {10.1016/j.jmaa.2019.03.052}
}

@article{zhang2022exponential,
  author  = {Zhang, J. and Fang, Z. and Sun, H.},
  title   = {Exponential-sum-approximation technique for variable-order time-fractional diffusion equations},
  journal = {Journal of Applied Mathematics and Computing},
  volume  = {68},
  number  = {1},
  pages   = {323--347},
  year    = {2022},
  doi     = {10.1007/s12190-021-01528-7}
}

@article{zhang2022fast,
  author  = {Zhang, J. and Fang, Z. and Sun, H.},
  title   = {Fast second-order evaluation for variable-order {C}aputo fractional derivative with applications to fractional sub-diffusion equations},
  journal = {Numerical Mathematics: Theory, Methods and Applications},
  volume  = {15},
  number  = {1},
  pages   = {200--226},
  year    = {2022},
  doi     = {10.4208/nmtma.OA-2021-0148}
}

@article{zheng2025two,
  author  = {Zheng, X.},
  title   = {Two methods addressing variable-exponent fractional initial and boundary value problems and {Abel} integral equation},
  journal = {CSIAM Transactions on Applied Mathematics},
  volume  = {6},
  number  = {4},
  pages   = {666--710},
  year    = {2025},
  doi     = {10.4208/csiam-am.SO-2024-0052}
}

@article{zheng2021optimal,
  author  = {Zheng, X. and Wang, H.},
  title   = {Optimal-order error estimates of finite element approximations to variable-order time-fractional diffusion equations without regularity assumptions of the true solutions},
  journal = {IMA Journal of Numerical Analysis},
  volume  = {41},
  number  = {2},
  pages   = {1522--1545},
  year    = {2021},
  doi     = {10.1093/imanum/draa013}
}

@article{huang2020high,
  title={High order symmetric direct discontinuous Galerkin method for elliptic interface problems with fitted mesh},
  author={Huang, H.Y. and Li, J. and Yan, J.},
  journal={Journal of Computational Physics},
  volume={409},
  pages={109301},
  year={2020},
}

@article{hao2026interior,
  author  = {Hao, C.C. and Huang, H.Y. and Zhang, H.L.},
  title   = {Interior penalty method for variable-order time-fractional mobile-immobile model with discontinuous coefficients},
  journal = {Journal of Applied Mathematics and Computing},
  volume  = {72},
  number  = {1},
  pages   = {16},
  year    = {2026}, 
}

@article{huang2026determining,
  author  = {Huang, H.Y. and Zheng, X.C. and Zhang, H.L.},
  title   = {Determining superconvergence points for {$L2-1_{\sigma}$} scheme of variable-exponent subdiffusion and error estimate},
  journal = {ZAMM - Journal of Applied Mathematics and Mechanics (Z Angew Math Mech)},
  volume  = {106},
  pages   = {e70326},
  year    = {2026}
}

@article{bai2023coupling,
  title={Coupling of direct discontinuous Galerkin method and natural boundary element method for exterior interface problems with curved elements},
  author={ Bai, S.Y. and Huang, H.Y.},
  journal={Advances in Computational Mathematics},
  volume={49},
  number={1},
  pages={6},
  year={2023},
}

@article{huang2020optimal,
  title={Optimal spatial $H1$-norm analysis of a finite element method for a time-fractional diffusion equation},
  author={Huang, C.B. and Stynes, M.},
  journal={Journal of Computational and Applied Mathematics},
  volume={367},
  pages={112435},
  year={2020},
}

@article{sun2019review,
  title={A review on variable-order fractional differential equations: mathematical foundations, physical models, numerical methods and applications},
  author={Sun, H.G. and Chang, A.L. and Zhang, Y. and Chen, W.},
  journal={Fractional Calculus and Applied Analysis},
  volume={22},
  number={1},
  pages={27--59},
  year={2019},
  publisher={De Gruyter}
}

@article{chen2022immersed,
  title={Immersed finite element method for time fractional diffusion problems with discontinuous coefficients},
  author={Chen, Y.P. and Li, Q.F. and Yi, H.M. and Huang, Y.Q.},
  journal={Computers \& Mathematics with Applications},
  volume={128},
  pages={121--129},
  year={2022},
  publisher={Elsevier}
}

@article{cangiani2018adaptive,
  title={Adaptive discontinuous Galerkin methods for elliptic interface problems},
  author={Cangiani, A. and Georgoulis, E. and Sabawi, Y.},
  journal={Mathematics of Computation},
  volume={87},
  number={314},
  pages={2675--2707},
  year={2018}
}

@article{cangiani2021version,
  title={$hp$-Version discontinuous Galerkin methods on essentially arbitrarily-shaped elements},
  author={Cangiani, A.  and  Dong, Z.N. and  Georgoulis, E. },
  journal={Mathematics of Computation},
  year={2021},
  volume={91},
  number={333},
  pages={1-35},
  doi = {10.1090/mcom/3667},
}

@inproceedings{3,
  title={Bone marrow transplantation in severe combined immunodeficiency with an unrelated MLC compatible donor},
  author={DuPont, B},
  booktitle={Proceedings of the third annual meeting of the International Society for Experimental Hematology. Houston: International Society for Experimental Hematology},
  volume={44},
  year={1974}
}

@book{riviere2008discontinuous,
  title={Discontinuous Galerkin methods for solving elliptic and parabolic equations: theory and implementation},
  author={Rivi{\`e}re, B.},
  year={2008},
  publisher={SIAM}
}

\end{document}